\documentclass[a4paper, 11pt, reqno]{amsart}
\usepackage{amsmath}
\usepackage{amsfonts}
\usepackage{amsthm}
\usepackage[utf8]{inputenc}
\usepackage{latexsym, amssymb}
\usepackage{epsfig}
\usepackage[pdftex]{hyperref}
\usepackage{indentfirst}
\usepackage{verbatim}
\usepackage[titletoc]{appendix}
\usepackage{url}
\usepackage{tikz}
\usetikzlibrary{matrix}

\newtheorem{thm}{Theorem}

\newtheorem{lem}{Lemma}
\newtheorem{rem}{Remark}
\newtheorem{prop}{Proposition}

\newtheorem*{defin}{Definition}
\newtheorem*{conj}{Conjecture}

\newcommand{\rd}{\partial}

\newcommand{\re}{\mathbb{R}}
\newcommand{\rp}{\mathbb{R}^{2}}
\newcommand{\hp}{\mathbb{H}^{2}}

\newcommand{\rt}{\mathbb{R}^{3}}
\newcommand{\sd}{\mathbb{S}^{2}}
\newcommand{\st}{\mathbb{S}^{3}}
\newcommand{\z}{\mathbb{Z}}
\newcommand{\si}{\Sigma}

\newcommand{\ric}{\operatorname{Ric}}
\newcommand{\nil}{\operatorname{Nil_3}}

\newcommand{\psl}{\operatorname{\widetilde{SL}_2}}

\title{On a conjecture of Meeks, Pérez and Ros}
\author{Vanderson Lima}
\address{Universidade Federal do Rio Grande do Sul\\
  Instituto de Matem\'atica e Estat\'istica\\
 Porto Alegre, RS - 91509-900, Brazil}
\email{vanderson.lima@ufrgs.br}

\begin{document}

\maketitle

\begin{abstract}

Meeks, Pérez and Ros conjectured that a closed Riemannian $3$-manifold which does not admit any closed embedded minimal surface whose two-sided covering is stable must be diffeomorphic to a quotient of the $3$-sphere. We give a counterexample to this conjecture. Also, we show that if we consider immersed surfaces instead of embedded ones, then the corresponding statement is true.

\end{abstract}

\providecommand{\abs}[1]{\lvert#1\rvert}

\linespread{1} 

\section{Introduction}

A closed minimal surface $\si$ immersed in a Riemannian $3$-manifold $(M,g)$ is called stable if the second variation of area is non-negative for all smooth variations of $\si$. Existence of stable minimal surfaces can be obtained when the ambient space presents nontrivial topology; for instance, one can minimize area in the isotopy class of an incompressible surface (see \cite{MSY}). On the other hand, quotients of the $3$-sphere endowed with metrics of positive Ricci curvature do not admit two-sided, closed, stable, minimal surfaces. These facts suggest that the existence of closed stable minimal surfaces is related to the topology of the ambient space. In \cite{M.P.R}, Meeks, P\'erez and Ros stated the following conjecture.

\begin{conj}[Meeks, P\'erez, Ros]
Let $(M,g)$ be a closed Riemannian $3$-manifold. If $(M,g)$ does not admit any closed, embedded, minimal surface whose two-sided covering is stable, then $M$ is finitely covered by the $3$-sphere.
\end{conj}

\begin{rem}
When the surface is one-sided, we consider the two-sided immersed surface associated to it, otherwise we consider the surface itself. 
\end{rem}

In this work we present a family of counter-examples to this conjecture. The examples are locally homogeneous Riemannian manifolds modelled on the {\it Thurston geometries} $\nil$ and $\psl(\re)$. The construction uses a topological characterization of these spaces and the classification of its stable minimal surfaces (due to Pitts-Rubinstein \cite{P.R}). We next state our first main result.

\begin{thm}\label{cx}
Let $(\widetilde{M},\widetilde{h})$ be $\nil$ or $\psl(\re)$ endowed with a homogeneous metric $\widetilde{h}$ whose isometry group has dimension $4$. There are closed Riemannian $3$-manifolds $(M,h)$ obtained as quotients of $(\widetilde{M},\widetilde{h})$ and which do not contain any closed, embedded, minimal surface whose two-sided cover is stable.
\end{thm}

We prove Theorem 1 in Section 3, using a family of examples presented in Section \ref{mainexample}.

One can wonder about what is the correct statement to the conjecture of Meeks, P\'erez and Ros. In this direction, we prove the following.

\begin{thm}\label{ims}
Let $(M,g)$ be a closed Riemannian $3$-manifold. If $(M,g)$ does not admit any closed, two-sided, immersed, stable minimal surface, then $M$ is finitely covered by the $3$-sphere.
\end{thm}

The proof of Theorem 2 uses results of existence of area-minimizing surfaces and some of the developments in the theory of $3$-manifolds of the last years, in particular the work of Perelman on the {\it Geometrization conjecture} \cite{Pe}, and the work of Kahn-Markovic on the {Surface subgroup conjecture} \cite{K.M}. 

In relation to the existence of the closed geometric 3-manifolds in Theorem \ref{cx} that provide
counterexamples to the Meeks, Perez Ros conjecture, we ask the following question, which we
highlight is not even known in the special case that the metric g is hyperbolic; see the end of
Section 4 for further discussion and motivation for this question.\\

\noindent
{\bf Question}: Let $M$ be a closed, non-Haken hyperbolic $3$-manifold and let $g$ be an arbitrary Riemannian metric on $M$. Does $(M,g)$ admit a closed, embedded, minimal surface whose two-sided covering is stable?\\

\noindent
{\bf Acknowledgments}: It is a pleasure to thank Lucas Ambrozio and Harold Rosenberg for discussions and suggestions on the manuscript. I also thank Laurent Mazet for a discussion about the main results, and Joaquin P\'erez for personal correspondence about this work. Finally, I would like to thank the anonymous referees by suggestions ans corrections.

\section{Seifert Fibered Spaces}\label{SFS}

In this section, we recall some facts about Seifert Fibered Spaces. The main references used here are \cite{Ma,S}. Let $D$ denote the unit disc of the complex plane centered at the origin. Let $p,q$ be two coprime integers with $p>0$ and $0 \leq q \leq p/2$. A \emph{standard fibered solid torus} with coefficients $(p,q)$ is an open solid torus
$$\frac{D \times [0,1]}{\bigl\{(z,0) \sim (\psi_{p,q}(z),1)\bigr\} }$$
where $\psi_{p,q}\colon D  \to D $ is given by $\psi_{p,q}(z) = e^{(2\pi iq/p)}z$. The fibration by \emph{vertical} segments $\{x\} \times [0,1]$ extends to a fibration by circles of the solid torus. The central fiber obtained by identifying the endpoints of $\{0\} \times [0,1]$ is the {\it core} of the solid torus, and every non-central fiber winds $p$ times around the core of $M$. The positive number $p$ is the \emph{multiplicity} of the central fiber. If $p=1$ the fibered solid torus is diffeomorphic to the usual product $D\times S^1$ and the central fiber is called \emph{regular}. If $p > 1$, the central fiber is called \emph{singular}.

A \emph{Seifert Fibered Space} is a 3-manifold $M$ foliated by circles, such that every circle has a fibered neighbourhood diffeomorphic to a standard fibered solid torus. Let $S$ be the topological space obtained from $M$ by quotienting the circles to points. Then, $S$ is a connected surface and has a natural orbifold structure: if the preimage of $x\in S$ is a fiber of multiplicity $p$, we see $x$ as a cone point of angle $\frac{2\pi}{p}$. The quotient map $\Pi: M \to S$ is called the Seifert fibration (one should think of the total space of this fibration as a circle bundle over the orbifold $S$). A Seifert fibration without singular fibers is a circle bundle over a surface, in the usual sense.

Let $\si \subset M$ be a closed, embedded surface in a Seifert Fibered space $M$. We say $\si$ is \emph{vertical} if it is the union of regular fibers. In this case we have two possibilities: $\si$ is either a torus or a Klein bottle, and in both cases the projection $\Pi(\si)$ to the base surface is a simple closed curve contained in the complement of the cone points of $S$. We also say that $\si$ is \emph{horizontal} if it is everywhere transverse to the fibers. On a closed Seifert manifold $M$ there is an invariant associated to it, called the Euler number and denoted by $e(M)$, which detects the presence of horizontal surfaces: $M$ contains a \emph{horizontal} surface if, and only if, $e(M) = 0$, see \cite{Ma}[Section 10.4.2].

\subsection{Geometric metrics}

Let $\mathbb{E}$ denote one of the fibered {\it Thurston geometries}, i.e., $\rt$, $\st$, $\sd\times\re$, $\hp\times\re$, $\nil$, $\psl(\re)$ (we will denote $\psl(\re)$ by $\psl$). Each of these spaces admits a homogeneous metric with isometry group of dimension at least $4$ and these metrics satisfy the following properties: there is a Riemannian submersion 
$\Pi^{\prime}: \mathbb{E} \rightarrow \mathbb{M}^{2}(\kappa),$ 
where $\mathbb{M}^{2}(\kappa)$ is the complete, simply connected Riemannian surface with constant curvature $\kappa$, and the fibers of $\Pi^{\prime}$ are the integral curves of a unit Killing vector field $\xi$. We have $\kappa = 0$ in the case of $\rt$ and $\nil$, $\kappa = 1$ in the case of $\sd\times\re$ and $\st$, and $\kappa = -1$ in the case of $\hp\times\re$ and $\psl$. Moreover each geometry has the structure of a line or circle bundle over $\mathbb{M}^{2}(\kappa)$, such that the bundle projection coincides with the submersion $\Pi^{\prime}$, and all the isometries of these metrics preserve this bundle structure.

Every Seifert Fibered Space $M$ is diffeomorphic to a quotient of some $\mathbb{E}$ as above, and conversely, every manifold which is a quotient of some $\mathbb{E}$ has a Seifert Fibered structure. Hence, we see that each Seifert fibered space has a locally homogeneous metric, and its Seifert fibration is induced by the bundle structure on the universal cover and the projection $\Pi^{\prime}$. Thus, away from the singular fibers, we have an $SO(2)$-action on the fibers by isometries and the projection $\Pi: M\to S$ is a Riemannian submersion, where the base orbifold is endowed with a constant curvature metric with isolated cone singularities. We call such a metric a {\it geometric metric}.

\begin{rem}
Two closed $3$-manifolds with different Thurston geometries are not diffeomorphic (see ~\cite[Theorem~5.2]{S}). This fact will be crucial to our counterexample in the next Section. 
\end{rem}

Let $M$ be a closed Seifert Fibered Space with base orbifold $S$. The underlying geometry for $M$ is determined in terms of its Euler number $e(M)$ and the orbifold Euler characteristic of $S$ (denoted $\chi(S)$), as given by the following table.

\vspace{0.2cm}

\begin{table}[!h]
\centering
\begin{tabular}[c]{lcccccr}
\hline
 & & $\chi < 0$ & & $\chi = 0$ & & $\chi > 0$ \\\\
$e = 0$ & & $\hp\times\re$ & & $\rt$ & &$\sd\times\re$ \\\\
$e \neq 0$ & & $\psl$ & & $\nil$ & & $\st$ \\
\hline
\label{The geometry of closed Seifert manifolds}
\end{tabular}
\end{table}

\subsection{Main example}\label{mainexample}

Let $p_1, \ p_2, \ p_3 \geqslant 2$ be natural numbers and let $\Delta$ be a geodesic triangle either in $\rp$ or in $\hp$, with inner angles $\frac{\pi}{p_1}$, $\frac{\pi}{p_2}$, $\frac{\pi}{p_3}$. By reflecting iteratively $\Delta$ along its sides we get a tessellation $T$. The \emph{triangle group} $\Gamma(p_1,p_2,p_3)$ is the group of isometries of $\rp$ or of $\hp$, generated by reflections along the three sides of $\Delta$. This group acts freely and transitively on the triangles of the tessellation $T$, hence it is discrete and $\Delta$ is a fundamental domain for $\Gamma(p_1,p_2,p_3)$. Consider the index-two subgroup $\Gamma = \Gamma^{\rm or}(p_1,p_2,p_3) \triangleleft \Gamma(p_1,p_2,p_3)$ of orientation-preserving isometries. Taking the quotient of the model space by $\Gamma$ we obtain an orbifold $(S,g)$, which is a sphere with exactly $3$ cone points of orders $p_1$, $p_2$ and $p_3$. $(S,g)$ is hyperbolic, or flat, according to whether $\frac{1}{p_1} + \frac{1}{p_2} + \frac{1}{p_3}$ is smaller than $1$, or equal to $1$, respectively.

Assume that $\frac{1}{p_1} + \frac{1}{p_2} + \frac{1}{p_3} < 1$. Recall that the geometry of $\psl$ can be constructed in the following way: consider the unit tangent bundle $U\hp$ of $\hp$ endowed with the {\it Sasaki metric}, and then take the universal cover of $U\hp$ endowed with the pullback metric. 

For each isometry $f$ of $\hp$, the differential $f_{*}$ is an isometry of $U\hp$, and therefore using the covering $\psl \to U\hp$, $f_{*}$ lifts to an isometry of $\psl$. So, there is a group of isometries of $\psl$, denoted by $\widetilde{\Gamma}$, which is an extension of $\Gamma$ by the central $\z$ subgroup of $\mbox{\em Isom}(\psl)$: we have the exact sequence
$$0 \to \z \to \widetilde{\Gamma} \to \Gamma \to 0.$$ 
Note that the action of $\widetilde{\Gamma}$ has no fixed points. In order to prove this, it suffices to show that for any $f\in \Gamma, f_*$ has no fixed points, unless $f$ is the identity. But the construction of $\Gamma$ gives that either $f$ is conjugate to a rotation or to a hyperbolic translation, from where it follows the claim. 

Let $M = \psl/\widetilde \Gamma$. Then, it is easy to prove that $M$ is a closed Seifert Fibered Space projecting on $(S,g)$. Moreover, $e(M)\neq 0$ and $M$ has 3 singular fibers whose respective fibered solid tori have coefficients $(p_i,1)$, $i=1,2,3$. Geometrically we can see $M$ as the unit tangent bundle of the orbifold $(S,g)$. In general, it is possible to construct other circle bundles over $(S,g)$ where $e \neq 0$ and at least one of the singular fibers has coefficients $(p_i,q_i)$ satisfying $q_i \neq 1$. Similar constructions in the case $\frac{1}{p_1} + \frac{1}{p_2} + \frac{1}{p_3} = 1$ produces quotients of $\nil$. For more details on these examples see \cite{S}.

\section{The counterexample}\label{main}

The next two results together with the main Example \ref{mainexample} are the keys steps in the construction of our counterexample in Theorem \ref{coex}.

\begin{lem}[Pitts, Rubinstein, \cite{P.R}]\label{PR}
Let $\si$ be a closed two-sided embedded stable minimal surface in a Seifert Fibered Space $M$ endowed with a geometric metric $h$. Then $\si$ is vertical or horizontal.
\end{lem}

\begin{proof}
If $M$ is a Lens space, then the metric $h$ has positive Ricci curvature, so in this case there is no closed, two-sided, stable minimal surface.
So, suppose $M$ is not a Lens space, and let $\Pi: M \to S$ be the Seifert Fibration. In this case, there exist a finite-degree covering $F:\widetilde{M} \to M$, an orbifold covering $f:\widetilde S \to S$ and a Seifert Fibration $\widetilde\Pi:\widetilde{M} \to \widetilde S$, such that $\widetilde S$ is a smooth surface, $\widetilde{M}$ is a $\mathbb{S}^1$-bundle over $\widetilde{S}$ and $\Pi\circ F = f\circ\widetilde\Pi$ (see ~\cite[Section~10.3.7]{Ma}). Also, $\widetilde h = F^{*}h$ is a geometric metric, and $\widetilde S$ and $S$ are endowed with respective metrics $\widetilde g$ and $g$ of (the same) constant curvature.

Consider a lift $\widetilde{\si}$ of $\si$ to $\widetilde{M}$. So, $\widetilde{\si}$ is a closed, two-sided, embedded, stable minimal surface in $(\widetilde M,\widetilde h)$. Let $\xi$ be the unit Killing vector field associated to $\widetilde\Pi$ and let $N$ be a unit normal to $\widetilde\si$. Denote $\phi = \langle N,\xi\rangle$. Suppose $\widetilde\si$ is neither horizontal nor a union of fibers. Then $\phi$ is not identically zero, but $\phi(x) = 0$ for some $x \in \widetilde\si$. Since $\xi$ is a Killing vector field we have 
\begin{equation}\label{jacobi}
L(\phi) = 0,
\end{equation}
where $L$ is the Jacobi operator of $\widetilde\si$. Since $\widetilde\si$ is stable and equation \eqref{jacobi} holds, by standard elliptic theory we have that $\phi$ does not vanish anywhere, or it is identically zero, which contradicts what we established before. Thus, $\widetilde\si$ is horizontal or it is a union of fibers. Since these two properties are local, the same conclusion holds for $\si$. 

Let $\si$ be a union of fibers, and suppose it contains a singular fiber. Thus $\gamma = \Pi(\si)$ is a simple closed curve containing a cone point $x \in S$ of angle $2\pi/p$, for some integer $p > 1$. Consider a neighbourhood $D$ of $x$ such that $f^{-1}(D)$ is a collection of disjoint geodesic disks $D_1,\cdots,D_m$ of $(\widetilde S,\widetilde g)$. Each $D_i$ is decomposed into $2p$ circular sectors of angle $\pi/p$, and $\widetilde\gamma = f^{-1}(\gamma\cap D)$ can be described as follows: on each sector there are two arcs meeting (uniquely) at the center of $D_i$. Hence, $\widetilde\Pi^{-1}(\widetilde\gamma)$ is not embedded. Since $\Pi\circ F = f\circ\widetilde\Pi$, we conclude that $\si$ is not embedded, which is a contradiction. Therefore, $\si$ cannot contain a singular fiber.
\end{proof}

\begin{prop}\label{geod}
Let $(S,g)$ be a two-sphere endowed with a flat or hyperbolic metric $g$ with three cone points. Then, $(S,g)$ does not contain a simple closed geodesic inside its regular part.
\end{prop}

\begin{proof}
We argue by contradiction and suppose there is such a geodesic $\gamma$. Consider disks $D_i$ around the cone points $x_i$, $i = 1, 2, 3$, such that $\gamma \subset \tilde{S} = S\backslash(\cup_{i} D_i)$. We can choose $D_i$ as the quotient of a geodesic disc $\tilde{D}_i$ in $\rp$ or $\hp$ by a rotation of angle $2\pi/p_i$, for some integer $p_i > 1$, where the center of this disc projects on the cone point. The boundary $\rd D_i$ is the projection of $\rd\tilde{D}_i$, thus the boundary components of $\tilde{S}$ have negative geodesic curvature with respect to the unit normal vector pointing towards $\tilde{S}$. 

Since the underlying surface is a sphere, the curve $\gamma$ separates $S$. Suppose $\gamma$ is the boundary of a disc $U$ on $\tilde{S}$. Then, by the Gauss-Bonnet formula and the hypothesis on the curvature of the orbifold we have 
\begin{equation}
2\pi = \int_{U} K_{S} \ dA + \int_{\gamma} \kappa_{g} \ dL = \left\{
\begin{array}{rl}
0, & \text{if } S \ \textrm{is flat},\\\\
-|U|, & \text{if } S \ \textrm{is hyperbolic};
\end{array} \right.
\end{equation}
which is a contradiction.

Thus, $\gamma$ must be homotopic to a the boundary components $\tilde{\gamma}$ of $\tilde{S}$. Let $\mathcal{A}$ be the annulus in $\tilde S$ bounded by $\gamma$ and $\tilde{\gamma}$. Since $\tilde\gamma$ has negative geodesic curvature we obtain
\begin{equation}
0 = \int_{\mathcal{A}} K_{S} \ dA + \int_{\gamma} \kappa_{g} + \int_{\tilde\gamma} \kappa_{g} \ dL < 0,
\end{equation}
which is again a contradiction.
\end{proof}

Let us recall some definitions concerning $3$-manifolds.
\begin{defin}
$1)$ We say that a $3$-manifold $M$ is irreducible if every $2$-sphere embedded in $M$ bounds an embedded $3$-ball in $M$. $M$ is $P^2$-irreducible if it is irreducible and contains no embedded two-sided projective plane.\\\\
$2)$ We say that a closed oriented $3$-manifold $M$ is Haken if it is irreducible and contains an embedded, two-sided, incompressible surface $\si$ (i.e., the map $\pi_1(\si) \to \pi_1(M)$ induced by the embedding is one-to-one) of genus $n \geq 1$. Otherwise, we say $M$ is non-Haken.
\end{defin}

\begin{thm}\label{coex}
Let $M$ be a closed, irreducible, non-Haken Seifert Fibered Space with infinite fundamental group, and let $h$ be a geometric metric on $M$. Then $(M,h)$ does not contain any closed, embedded, minimal surface whose two-sided cover is stable.
\end{thm}

\begin{proof}
It is well-known that a closed, irreducible, non-Haken Seifert Fibered Space $M$ with infinite fundamental group has $e(M) \neq 0$ and the base orbifold $S$ is a sphere with three cone points (see for example Proposition $2$ in \cite{H}). The possible geometric structures for $M$ are $\nil$ and $\psl$, and for the geometric metrics on $M$, $S$ has a flat or hyperbolic metric. So, $M$ belong to the class of manifolds discussed in the Main example \ref{main}.

We argue by contradiction and suppose that there is $\si \subset M$, a closed embedded minimal surface whose two-sided covering is stable. In the case $\si$ is two-sided, by Lemma \ref{PR}, $\si$ is either vertical or horizontal. If $\si$ is vertical then its projection $\gamma$ on the base surface is a simple closed curve in the regular part of $S$. We next prove that $\gamma$ is a geodesic.

Let $T$ and $\eta$ be respectively a unit tangent field and a unit normal field to $\gamma$. Given a point $x \in \si$, consider an small neighborhood $U$ of $\Pi(x)$ in $S$ such that the projection $\Pi: \Pi^{-1}(U) \to U$ is a Riemannian submersion. Let $\widetilde{T}$ and $\widetilde{\eta}$ be the corresponding horizontal lifts to $\Pi^{-1}(U)$ of $T$ and $\eta$. Then, $\{\widetilde{T},\xi\}$ is an orthonormal basis on $T_{x}\si$ and $\widetilde{\eta}$ a unit normal to $\si$, where $\xi$ is the unit Killing vertical vector field of $\Pi^{-1}(U)$. Also, let $\overline{\nabla}$ and $\nabla$ denote the connections on $M$ and $S$ respectively. Then, the mean curvature of $\Sigma$ at x is
\begin{align*}
0&=-\left\langle\overline{\nabla}_{\widetilde{T}}\widetilde{\eta},\widetilde{T}\right\rangle-\left\langle \overline{\nabla}_{\xi}\widetilde{\eta},\xi\right\rangle =\left\langle \overline{\nabla}_{\widetilde T} \widetilde T,\widetilde{\eta}\right\rangle=\left\langle \nabla_T T,\eta\right\rangle\\
&=\rm{the\:geodesic\;curvature\;of\;\gamma\;at\;the\;point\;\Pi(x)}.
\end{align*}

However, by Proposition \ref{geod}, $S$ does not admit simple closed geodesics in its regular part. On the other hand, as we mentioned in the Section \ref{SFS}, closed manifolds with the geometries of $\nil$ and $\psl$ do not contain horizontal surfaces. So, in any case we obtain a contradiction. 

Finally, if $\si$ is one-sided we can pass to a double covering $F: \widetilde{M} \to M$ so that the lift $\widetilde\si$ of $\si$ to $\widetilde{M}$ is a connected closed embedded two-sided minimal surface (see ~\cite[Proposition 3.7]{Z}), which is stable by hypothesis. The manifold $\widetilde{M}$ is also a Seifert Fibered space, endowed with the geometric metric $\widetilde{h} = F^{*}h$, so again by Lemma \ref{PR}, $\widetilde\si$ is either vertical or horizontal, thus this also holds for $\si$, and the contradiction follows as before.
\end{proof}

\section{Further discussion}\label{last}

In this section, we prove Theorem \ref{ims} and motivate the question we presented in the Introduction.

\begin{proof}[Proof of Theorem \ref{ims}]
Suppose $M$ is not a quotient of the $3$-sphere. We will prove that $(M,g)$ contains a two-sided, immersed, stable minimal surface. In this case, it follows from the work of Perelman that $M$ has infinite fundamental group, thus there are four possibilities:\\

\noindent
{\bf 1) $M$ is not irreducible}: Then, there is an embedded sphere $S \subset M$, which represents a non zero element of $\pi_{2}(M)$. So, combining the results in \cite{S.U1} and \cite{M.Y}, there is an immersed stable minimal sphere $\si$ in $(M,g)$, and either $\si$ is embedded or it double covers an embedded projective plane.

\noindent
{\bf 2) $M$ is orientable, irreducible and Haken}: Combining the results in \cite{S.U2,S.Y} with that of \cite{F.H.S}, $(M,g)$ admits an incompressible, orientable, stable minimal surface of genus $n \geq 1$. Moreover, this surface is either embedded or it double covers an embedded non-orientable surface. 

\noindent
{\bf 3) $M$ is orientable, irreducible and non-Haken}: We have two sub-cases. If $\pi_1(M)$ contains $\mathbb{Z}\oplus\mathbb{Z}$ as a subgroup, then there is an immersed incompressible torus in $M$. Thus it follows from the results in \cite{S.U2,S.Y} that $(M,g)$ contains an immersed stable minimal torus; if $\pi_1(M)$ does not contains any subgroup isomorphic to $\mathbb{Z}\oplus\mathbb{Z}$, then by the work of Perelman, $M$ admits a hyperbolic metric (which can be different from the metric $g$). It follows from the work of Kahn-Markovic \cite{K.M} that $M$ admits an orientable, incompressible immersed surface of genus $n \geq 2$. So using again \cite{S.U2,S.Y} we conclude that $(M,g)$ contains an immersed stable minimal surface of genus $n \geq 2$.

\noindent
{\bf 4) $M$ is non-orientable and irreducible}: If $M$ contains an embedded $2$-sided projective plane $P$, then by ~\cite[Proposition $2.3$]{BBEN}, $(M,g)$ contains a embedded stable minimal surface $\si$ homeomorphic to $P$. Moreover $\si$ is $2$-sided (otherwise, by ~\cite[Lemma 2]{H2} $M$ is $\mathbb{R}P^3$, which contradicts the fact that $M$ is non-orientable). If $M$ is $P^{2}$-irreducible, by Lemmas $6.6$ and $6.7$ in \cite{He}, $M$ contains an embedded, $2$-sided nonseparating incompressible surface $S$. Using Theorems $3.1$ and $5.1$ of \cite{F.H.S} we obtain a stable minimal surface $\si$ homotopic to $S$ in $(M,g)$, moreover, $\si$ is embedded or double covers an embedded one-sided surface.
\end{proof}

Analysing the previous proof, we see that if $M$ is either non-orientable, or not irreducible, or irreducible and Haken, then $(M,g)$ contains an embedded minimal surface whose two-sided covering is stable. If $M$ is irreducible, non-Haken and $\pi_1(M)$ contains $\mathbb{Z}\oplus\mathbb{Z}$ as a subgroup, it follows from the {\it Seifert Fibered Space theorem} (see Section $2.3.2$ in \cite{Pr}) that $M$ is a non-Haken Seifert Fibered Space with infinite $\pi_1$, and as we saw in the Theorem \ref{coex}, these spaces admit a metric which does not contain closed, embedded minimal surfaces whose two-sided cover is stable. So the only case remaining to study is that when $M$ is a non-Haken hyperbolic manifold. There are examples of non-Haken hyperbolic manifolds which do not contain any non-orientable embedded surface, so in this case the surfaces produced by the last theorem are not embedded and do not cover any non-orientable surface.\\

\noindent
{\bf Question}: Let $M$ be a closed, non-Haken hyperbolic $3$-manifold and let $g$ be an arbitrary Riemannian metric on $M$. Does $(M,g)$ admit a closed, embedded, minimal surface whose two-sided covering is stable?\\

It is important to highlight that this question is open even when the metric $g$ is the hyperbolic metric.

\end{document}